\documentclass{article}

\usepackage{amsmath,verbatim}
\usepackage{amsfonts}
\usepackage{mathrsfs}
\usepackage{amsthm}
\usepackage{amssymb}
\usepackage{mathtools}
\usepackage{blindtext}
\usepackage[utf8]{inputenc}
\usepackage[margin=1in]{geometry}
\usepackage{url}

\newtheorem{claim}{Claim}

\theoremstyle{definition}

\newtheorem{theorem}{Theorem}
\newtheorem{proposition}[theorem]{Proposition}
\newtheorem{corollary}[theorem]{Corollary}
\newtheorem{lemma}[theorem]{Lemma}

\newcommand{\Aut}{\text{Aut}}

\title{Smallest cyclically covering subspaces of $\mathbb{F}_q^n$,\\
and lower bounds in Isbell's conjecture}
\author{ Peter Cameron\\
\and
David Ellis\\
\and
William Raynaud
}
\date{8th October 2018}

\begin{document}
	\maketitle

	\begin{abstract}

For a prime power $q$ and a positive integer $n$, we say a subspace $U$ of ${\mathbb{F}_q^n}$ is {\em cyclically covering} if the union of the cyclic shifts of $U$ is equal to $\mathbb{F}_q^n$. We investigate the problem of determining the minimum possible dimension of a cyclically covering subspace of $\mathbb{F}_q^n$. (This is a natural generalisation of a problem posed in 1991 by the first author.) We prove several upper and lower bounds, and for each fixed $q$, we answer the question completely for infinitely many values of $n$ (which take the form of certain geometric series). Our results imply lower bounds for a well-known conjecture of Isbell, and a generalisation theoreof, supplementing lower bounds due to Spiga. We also consider the analogous problem for general representations of groups. We use arguments from combinatorics, representation theory and finite field theory.
		
	\end{abstract}	

	\section{Introduction}
	
	For a prime power $q$, let $\mathbb{F}_q$ denote the finite field of order $q$. For $n \in \mathbb{N}$, let $\{e_1,e_2,...,e_n\}$ denote the standard basis for $\mathbb{F}_q^n$. Let $\sigma : \mathbb{F}_q^n \rightarrow \mathbb{F}_q^n$ be the linear map defined by $\sigma(\sum_{i=1}^{n}{x_i e_i}) = \sum_{i=1}^{n}{x_{i-1} e_i}$, where addition/subtraction in the index is modulo $n$. That is, $\sigma$ is the \textit{cyclic shift} operator which shifts each entry one place clockwise. Given a subspace $U \leq \mathbb{F}_q^n$ and a linear map $\alpha : \mathbb{F}_q^n \rightarrow \mathbb{F}_q^n$ we let $\alpha (U) = \{ \alpha (x) : x \in U\}$. In particular, for $r \in \{0,1,2,3,...,n-1\}$, $\sigma^r (U)$ is the subspace of $\mathbb{F}_q^n$ obtained by cyclically shifting the elements of $U$ precisely $r$ places clockwise. We call $\{ \sigma^r(U) : 0\leq r \leq n-1\}$ the family of \textit{cyclic shifts} of $U$.

	We say a subspace $U \leq \mathbb{F}_q^n$ is \textit{cyclically covering} if $\bigcup_{r=0}^{n-1}{\sigma^r (U)} = \mathbb{F}_q^n$. For a prime power $q$ and $n \in \mathbb{N}$, we define $h_q (n)$ to be the maximum possible codimension of a cyclically covering subspace of $\mathbb{F}_q^n$. The main purpose of this paper is to investigate the behaviour of the function $h_q : \mathbb{N} \rightarrow \mathbb{N}$, for various prime powers $q$.

	This problem is a natural generalisation of the following problem, posed by the first author. For $n \in \mathbb{N}$, define $V_n = \{x \in \mathbb{F}_2^n : \sum_{i=1}^{n}{x_i} = 0\}$, i.e.\ $V_n$ is the ${\mathbb{F}_2}$-vector space of binary strings with length $n$ and even Hamming weight. For an odd positive integer $n$, define $f(n)$ to be the maximum possible codimension of a subspace $W$ such that the union of the cyclic shifts of $W$ is equal to $V_n$. The first author asked in \cite[Problem 190]{cameronbcc} whether $f(n)$ tends to infinity as $n \to \infty$ (over odd integers $n$).

	We observe that $f(n) = h_2 (n)$ for all odd $n \in \mathbb{N}$. Indeed, take $W \leq V_n$ such that $V_n$ is equal to the union of the cyclic shifts of $W$. Then $W' := \text{Span}(W \cup \{11...1\})$ is a cyclically covering subspace of $\mathbb{F}_2^n$ with the same codimension as that of $W$ in $V_n$. Conversely, if $U \leq \mathbb{F}_2^n$ is a cyclically covering subspace, then the cyclic shifts of $U' := U \cap V_n$ cover $V_n$, and the codimension of $U'$ in $V_n$ is equal to the codimension of $U$ in $\mathbb{F}_2^n$.

	We remark that somewhat similar problems have been investigated before. In \cite{Luh}, for example, Luh shows that any vector space (finite or infinite) over $\mathbb{F}_q$ can be expressed as a union of $q+1$ proper subspaces, and that this expression is unique up to automorphisms of the vector space. In \cite{jamison}, Jamison determined, for each $0 < k < n$, the minimum number of $k$-flats that are required to cover $\mathbb{F}_q^n \setminus \{0\}$. (Here, a {\em $k$-flat} is a translate of a $k$-dimensional subspace.)

Our results have implications for a well-known conjecture of Isbell (and a generalisation thereof), as we now describe. (Indeed, this was the first author's original motivation for studying the above problem, though we also believe that the problem is natural in its own right.)

For $n \in \mathbb{N}$, we let $[n]: = \{1,2,\ldots,n\}$ denote the standard $n$-element set, and we write $S_n$ for the symmetric group on $[n]$. If $G \leq S_n$ is a permutation group, we say that $n$ is the {\em degree} of $G$, and we say that $G$ is {\em transitive} if for every $i,j \in [n]$, there exists $\sigma \in G$ such that $\sigma(i)=j$. If $X$ is a finite set, we write $\mathcal{P}(X)$ for the power-set of $X$. If $\mathcal{F} \subset \mathcal{P}([n])$, we say $\mathcal{F}$ is {\em intersecting} if any two sets in $\mathcal{F}$ have nonempty intersection, we say it is an {\em up-set} if whenever $S \in \mathcal{F}$ and $S \subset T$ we have $T \in \mathcal{F}$, and we say it is {\em antipodal} if for any $S \subset [n]$, $\mathcal{F}$ contains exactly one of $S$ and $[n] \setminus S$. If $\mathcal{F} \subset \mathcal{P}([n])$, we define its {\em automorphism group} by $\Aut(\mathcal{F}) : = \{\sigma \in S_n:\ \sigma(\mathcal{F}) = \mathcal{F}\}$, where $\sigma(\mathcal{F}): = \{\sigma(S):\ S \in \mathcal{F}\}$, and we say that $\mathcal{F}$ is {\em symmetric} if $\Aut(\mathcal{F})$ is a transitive subgroup of $S_n$.

Isbell \cite{isbell-1}, and later Frankl, Kantor and the first author \cite{cfk}, investigated the set
\begin{equation}\label{eq:A-defn} A := \{n \in \mathbb{N}: \text{there exists a symmetric intersecting family }\mathcal{F} \subset \mathcal{P}([n]) \text{ with } |\mathcal{F}|=2^{n-1}\}.\end{equation}
Since for any $S \subset [n]$, an intersecting family $\mathcal{F} \subset \mathcal{P}([n])$ contains at most one of $S$ and $[n] \setminus S$, we have $|\mathcal{F}| \leq 2^{n-1}$ for any intersecting $\mathcal{F}$. It is easy to see that an intersecting family $\mathcal{F} \subset \mathcal{P}([n])$ is maximal intersecting if and only if $|\mathcal{F}|=2^{n-1}$, and that a family $\mathcal{F} \subset \mathcal{P}([n])$ is maximal intersecting if and only if it is an antipodal up-set. Symmetric, antipodal up-sets arise naturally as the `winning sets' in $n$-player games where $n$ different players are choosing between two alternatives and their choices are aggregated according to some rule (with symmetry being a natural notion of the fairness of the rule), and Isbell \cite{isbell-1} was led to their study from problems in Social Choice Theory. (Indeed, Isbell termed a symmetric, antipodal up-set a {\em fair game}, though we do not use this terminology here, to avoid confusion with other notions of fair games.) Isbell \cite{isbell-2} observed that the set $A$ defined in (\ref{eq:A-defn}) is equal to the set of all positive integers $n$ for which there exists a transitive permutation group of degree $n$ having no fixed-point-free element of 2-power order. He conjectured that there exists a function $m:\{b \in \mathbb{N}: b \text{ odd}\} \to \mathbb{N}$ such that if $b \in \mathbb{N}$ is odd and $a \geq m(b)$, then $2^a \cdot b \notin A$. Isbell's conjecture remains open. Frankl, Kantor and the first author \cite{cfk} proved that for $b \in \{1,3\}$ one can take $m(1)=1$ and $m(3)=2$ (which is best possible). They also conjectured a generalisation of Isbell's conjecture, namely that if, for each prime $p$, we define $A_p$ to be the set of all positive integers $n$ for which there exists a transitive permutation group of degree $n$ having no fixed-point-free element of $p$-power order, then for each prime $p$, there exists a function $m_p :\{b \in \mathbb{N}:\ \text{gcd}(b,p)=1\} \to \mathbb{N}$ such that if $b \in \mathbb{N}$ is coprime to $p$ and $a \geq m_p(b)$, then $p^a \cdot b \notin A_p$. (Hereafter, we refer to this as the {\em generalised Isbell conjecture}.) For all primes $p$, the $p$-case of this conjecture remains open, although several related results have been proved; for example, using the Classification of Finite Simple Groups, Fein, Kantor and Schacher \cite{fks} proved that any transitive permutation group of degree $n>1$ contains a fixed-point-free element of $p$-power order for some prime $p$.

It is natural to ask for lower bounds on the function $m_p$ in the generalised Isbell conjecture. To obtain such a lower bound, it suffices to construct a transitive permutation group of degree $n=p^a \cdot b$, with $b$ coprime to $p$, containing no fixed-point-free element of $p$-power order,
and with $a$ large compared to $b$. One construction method
is to take the vector space $V=\mathbb{F}_p^b$, with $b$ coprime to $p$, and to take a subspace $W \leq V$ of smallest
possible dimension such that the cyclic shifts of $W$ (i.e., the images of $W$ under powers of the cyclic shift operator $\sigma$, defined above) cover $V$. Let $G: = V \rtimes C_b$ be the semidirect product of $V$ by the cyclic group $C_b = \langle \sigma \rangle$, and let $\iota: V \hookrightarrow G;\ v \mapsto (v,\text{Id})$ denote the natural inclusion map. Consider the permutation group $H$ induced by the left action of $G$ on the left cosets of $\iota(W)$ in $G$. It is easy to see that every element of $p$-power order in $H$ is induced by an element of $G$ of the form $(v,\text{Id})$ for some $v \in V$; such an element fixes some left coset of $\iota(W)$ in $G$, namely $\iota(\sigma^j(W))$, where $j \in \{0,1,\ldots,b-1\}$ is such that $v \in \sigma^j(W)$. Hence, any element of $H$ of $p$-power order has a fixed point. The degree of $H$ is $p^a \cdot b$, where $a:=\dim(V)-\dim(W) = h_p(b)$ is the codimension of $W$. This yields the following.

\begin{proposition}
\label{prop:isbell-connection}
For prime $p$ and each $b \in \mathbb{N}$ coprime to $p$, define $m_p(b) = \min\{c \in \mathbb{N}:\ p^a \cdot b \notin A_p \ \forall a \geq c\}$, with the usual convention that $m_p(b)=\infty$ if the set in question is empty. Then $m_p(b) > h_p(b)$ for all integers $b \in \mathbb{N}$ that are coprime to $p$.
\end{proposition}

This method led the first author to the problem considered in this paper. As we outline later, a construction due to Spiga \cite{spiga-thesis} (building on the work of Suzuki in \cite{suzuki}) gives a lower bound for $m_p(b)$ which is better than ours for certain integers $b$; on the other hand, our construction is simpler and works for certain natural infinite sequences of integers where Spiga's method does not apply.

The rest of this paper is structured as follows. In Section 2, we prove some simple upper and lower bounds on the function $h_q(n)$. In Section 3, we prove our main results, which are lower bounds on $h_q(n)$ that are sharp for infinitely many values of $n$. In Section 4, we generalise the problem considered here to arbitrary group representations, and we prove some straightforward upper and lower bounds for the general problem. In Section 5, we exhibit, for each prime power $q$, infinitely many values of $n$ for which $h_q(n)=0$; this result is obtained as a special case of a result for arbitrary group representations, where the only covering subspace is the whole space. 

	\section{Simple upper and lower bounds}

We first recall a straightforward lower bound on $h_2(n)$, due to the first author (unpublished). We give a proof for completeness.

	\begin{lemma}
	\label{lemma:simple}
		For odd positive integers $n > 3$, we have $h_2 (n) \geq 2$.
	\end{lemma}

	\begin{proof}

		Let $U = \text{Span}(S)$, where

		\begin{equation*}
			S = \{1111111...11, 1010000...00, 0001100...00,0000110...00,0000011...00,....,0000000...11\}
		\end{equation*}
Since $S$ is a linearly independent set of size $n-2$, we have $\text{codim}(U) = 2$. We claim that $U$ is a cyclically covering subspace of $\mathbb{F}_2^n$.  Observe that the last $n-4$ elements of $S$ are a basis for the subspace $\{x \in V_n : x_1 = x_2 = x_3 = 0\}$. First let $x \in \mathbb{F}_2^n$ have even Hamming weight. Since $x$ has an odd number of zeros, it has a (cyclic) interval of consecutive zeros, with odd length. In particular, it contains a (cyclic) interval of the form $000$ or $101$. By cycling $x$, we may assume that $x_1x_2 x_3 = 000$ or $x_1x_2x_3 = 101$. In the first case, $x$ lies in the span of the last $n-4$ elements of $S$; in the second, $x+101000\ldots0$ lies in the span of the last $n-4$ elements of $S$, so we are done. Now let $x \in \mathbb{F}_2^n$ have odd Hamming weight. Then $x+11\ldots1$ has even Hamming weight, and $11\ldots 1 \in S$, so again we are done.
\end{proof}

It is easy to check that $h_2(3)=1$, and therefore the assumption $n >3$ in Lemma \ref{lemma:simple} is necessary. Equality holds in Lemma \ref{lemma:simple} for $n=5$.
\bigskip
		
We next give a rather crude `product' bound.
\begin{lemma}
\label{lma:products}
If $q$ is a prime power, and $n,m \in \mathbb{N}$, then
		
\begin{equation*}
h_q(nm) \geq \max\{h_q (n),h_q (m)\}.
\end{equation*}
\end{lemma}

\begin{proof}
Let $q$ be a prime power. If $v$ is a vector in $\mathbb{F}_q^N$ for some $N \in \mathbb{N}$, let us write $v(j)$ for the $j$th component of $v$ (i.e., $v = \sum_{i=1}^{N}{v(i) e_i}$ with respect to the standard basis $\{e_1,e_2,...,e_N\}$).
		
Let $n,m \in \mathbb{N}$. Without loss of generality, we may assume that $h_q (n) \geq h_q (m)$. Let $U \leq \mathbb{F}_q^n$ be a cyclically covering subspace of $\mathbb{F}_q^n$, with $\textrm{codim}(U) = h_q (n)$. Let $k = h_q(n)$. Let $\{u_1,...,u_{n-k}\}$ be a basis for $U$. For each $i \in [n-k]$, let
\begin{equation*}
x_i = (\underbrace{0,0,...,0,u_i(1)}_m ,\underbrace{0,0,...,0,u_i(2)}_m ,...,\underbrace{0,0,...,0,u_i(n)}_m) \in \mathbb{F}_q^{nm}.
\end{equation*}
Let
\begin{equation*}
S = \{x_i:\ i \in [n-k]\} \cup \{e_j:\ j \in [nm],\ m \nmid j\},
\end{equation*}
where $e_j$ is the $j$th standard basis vector in $\mathbb{F}_q^{nm}$, and let $V = \textrm{Span}(S)$. Then $|S| = (n-k) + (nm - n) = nm-k$, and $S$ is linearly independent, so $\textrm{codim}(V)=k$. We claim that $V$ cyclically covers $\mathbb{F}_q^{nm}$. Indeed, if $x \in \mathbb{F}_q^{nm}$, then consider the projection of $x$ onto the subspace spanned by $\{e_j:\ m \mid j\}$, i.e.\
\begin{equation*}
\pi(x) := (\underbrace{0,0,\ldots,0,x(m)}_m,\underbrace{0,0,\ldots,0,x(2m)}_m,\underbrace{0,0,\ldots,0,x(3m)}_m,\ldots \underbrace{0,0,\ldots,0,x(nm)}_m) \in \mathbb{F}_q^{nm},
\end{equation*}
and let
\begin{equation*}
\psi(x) := (x(m),x(2m),x(3m),...,x(nm)) \in \mathbb{F}_q ^n
\end{equation*}
be the vector obtained from $\pi(x)$ by deleting the coordinates that are not multiples of $m$. Since $U$ cyclically covers $\mathbb{F}_q^n$, there exists $r \in [n-1]$ such that $\sigma^{r}(\psi(x)) \in U$. It follows that $\sigma^{mr}(\pi(x)) \in V$, and therefore $\sigma^{mr}(x) \in V$, since $S$ contains every unit vector $e_j$ such that $m \nmid j$. Hence, $V$ is cyclically covering, as claimed, and therefore $h_q (nm) \geq \textrm{codim}(V) = k = h_q (n)$, proving the lemma.

	\end{proof}

We now give a straightforward upper bound for $h_q(n)$, for all $n \in \mathbb{N}$.

\begin{lemma}
\label{lemma:simple-upper}
For $q$ a prime power and $n \in \mathbb{N}$, we have $h_q (n) \leq \left \lfloor {\log_{q}(n)} \right \rfloor$.
\end{lemma}
\begin{proof}

Let $U \leq \mathbb{F}_q^n$ be a cyclically covering subspace. The cyclic group $\left \langle \sigma \right \rangle = \{ \text{Id}, \sigma ,\sigma^2,\ldots \sigma^{n-1}\}$ acts on $\mathbb{F}_q^n$ by cyclically shifting vectors. The orbits of this group action partition $\mathbb{F}_q^n$, and each orbit contains at most $n$ vectors, so there are at least $q^n / n$ orbits. Since $U$ is cyclically covering, it intersects each orbit, and therefore $|U| \geq q^n /n$. Hence, $\text{dim}(U) = \log_q(|U|) \geq n - \log_q (n)$, so $\text{codim}(U) \leq \log_q (n)$, proving the lemma.
\end{proof}

Our main results show that for each prime power $q$, the simple upper bound in Lemma \ref{lemma:simple-upper} is sharp for infinitely many values of $n$. The proofs of these results occupy most of the next section.

\section{Our main results}

Our first main result is as follows.

\begin{theorem}
\label{thm:main}
If $q$ is a prime power and $d \in \mathbb{N}$, then

\begin{equation*}
h_q(q^d -1) = d-1 = \left \lfloor \log_q (q^d - 1) \right \rfloor.
\end{equation*}
\end{theorem}

Observe that the upper bound $h_q(q^d -1) \leq d-1$ is immediate from Lemma \ref{lemma:simple-upper}. Our proof of the lower bound $h_q(q^d -1) \geq d-1$ requires some standard facts from the Galois theory of finite fields, which we now briefly recall; the reader is referred to \cite{LN} for more background.

For a prime power $q$, we write $\overline{\mathbb{F}}_q$ for the algebraic closure of $\mathbb{F}_q$. We write $F_q:\overline{\mathbb{F}}_q \to \overline{\mathbb{F}}_q;\ x \mapsto x^q$ for the Frobenius automorphism of $\overline{\mathbb{F}}_q$. For $\omega \in \overline{\mathbb{F}}_q$, we call $\omega,F_q(\omega), F_q^2 (\omega), F_q^3(\omega),\ldots$ the {\em Galois conjugates} of $\omega$; if $\omega \in \overline{\mathbb{F}}_q$ is a root of some polynomial $f \in \mathbb{F}_q[X]$, then all the Galois conjugates of $\omega$ are also roots of $f$. Since any element $\omega \in \overline{\mathbb{F}}_q$ is the root of some polynomial in $\mathbb{F}_q[X]$, and all of the Galois conjugates of $\omega$ are roots of this polynomial, any $\omega \in \overline{\mathbb{F}}_q$ has only finitely many Galois conjugates. If $\omega \in \overline{\mathbb{F}}_q$, the \textit{minimal polynomial} of $\omega$ over $\mathbb{F}_q$ is the unique non-zero, monic polynomial in $\mathbb{F}_q[X]$ of minimal degree, that has $\omega$ as a root.

We will make repeated use of the following well-known fact (see for example Theorem 3.33 in \cite{LN}).
\begin{proposition}
\label{prop:min-poly}
Let $q$ be a prime power and let $\omega \in \overline{\mathbb{F}}_q$. Let $\omega,\omega^q,\ldots,\omega^{q^{t-1}}$ be the distinct Galois conjugates of $\omega$. Then
$$f(X) = \prod_{i=1}^{t} (X-\omega^{q^{t-1}})$$
is the minimal polynomial of $\omega$.
\end{proposition}

We are now ready to prove Theorem \ref{thm:main}.

\begin{proof}[Proof of Theorem \ref{thm:main}]
Let $q$ be a prime power, let $d \in \mathbb{N}$ and let $n = q^d - 1$. We identify $\mathbb{F}_q^n$ and $\mathbb{F}_q [X]/\left \langle X^n -1 \right \rangle$ via the linear isomorphism taking $v \in \mathbb{F}_q^n$ to the polynomial $\sum_{i=1}^{n} {v(i) X^{i-1}}$. The action of $\sigma$ on $\mathbb{F}_q^n$ then corresponds to multiplication by $X$ in $\mathbb{F}_q [X]/\left \langle X^n -1 \right \rangle$.

Since $X^n-1$ and $nX^{n-1} = \frac{d}{dX} (X^n - 1)$ are coprime, it follows that $X^n-1$ has no repeated roots in $\overline{\mathbb{F}}_q$.  Let
\begin{equation}\label{eq:factorization} \prod_{i=1}^{N} f_i (X) = X^n-1\end{equation}
be a factorization of $X^n-1$ into monic irreducible polynomials $f_i(X) \in \mathbb{F}_q [X]$. Since $X^n-1$ has no repeated roots in $\overline{\mathbb{F}}_q$, the $f_i (X)$ are distinct. Moreover, each pair $f_i (X), f_j(X)$ is coprime, since if $p(X) \neq 1$ is a monic common factor of $f_i(X)$ and $f_j(X)$ then by irreducibility, we have $p(X) = f_i(X) = f_j(X)$, and therefore $i=j$.

Define a linear map
\begin{equation}
\label{eq:theta}
\theta : \frac{\mathbb{F}_q [X]}{\left \langle X^n -1 \right \rangle} \rightarrow \bigoplus_{i=1}^{N} {\frac{\mathbb{F}_q [X]}{\left \langle f_i (X) \right \rangle}};\quad \theta (p(X)) = \left(p(X) \text{ mod } f_i (X)\right)_{i=1}^{N},
\end{equation}
i.e.,\ $\theta$ is the direct sum of the natural quotient maps corresponding to the ideals generated by each $f_i$. Since the $f_i (X)$ are pairwise coprime, it follows from the Chinese Remainder Theorem for rings that $\theta$ is a linear isomorphism. For each $i \in [N]$, define
$$V_i = \left\{p(X) \in \frac{\mathbb{F}_q [X]}{\left \langle X^n -1 \right \rangle}\ :\  \prod_{j \neq i}{f_j(X)} \text{ divides } p(X) \right\}.$$
Since for each $i \in [N]$, we have
$$V_i = \theta^{-1}\left(\{0\} \times \ldots \times \{0\} \times  \frac{\mathbb{F}_q [X]}{\left \langle f_i (X) \right \rangle} \times \{0\} \times \ldots \times \{0\}\right),$$
(where the zeros are in each place except for the $i$th) and $\theta$ is a linear isomorphism, we have the direct sum decomposition
\begin{equation}
\label{eq:decomp}
\frac{\mathbb{F}_q [X]}{\left \langle X^n -1 \right \rangle} = \bigoplus_{i=1}^{N} {V_i},
\end{equation}
and $V_i$ may be viewed as the copy of $\mathbb{F}_q [X]/ \langle f_i (X) \rangle$ in $ \mathbb{F}_q [X]/ \langle X^n -1 \rangle$, for each $i \in [N]$. Moreover, each $V_i$ is closed under multiplication by $X$ (i.e., under the cyclic action of $\sigma$ the $V_i$ are \emph{invariant} subspaces).

Since $\textrm{char}(\mathbb{F}_q) \nmid n$, there exists a primitive $n$th root of unity in $\overline{\mathbb{F}}_q$. Let $\omega \in \overline{\mathbb{F}}_q$ be one such. Since $n=q^d-1$, $q$ has multiplicative order $d$ modulo $n$, and so the iterates of $\omega$ under the Frobenius automorphism are precisely $\omega,\omega^q,\omega^{q^2},\ldots,\omega^{q^{d-1}}$ (and these are distinct). Hence, by Proposition \ref{prop:min-poly}, the minimal polynomial of $\omega$ over $\mathbb{F}_q$ is
\begin{equation}\label{eq:min-poly}
f(X) = (X-\omega)(X-\omega^q)(X-\omega^{q^2})...(X-\omega^{q^{d-1}}) \in \mathbb{F}_q [X].
\end{equation}
As $f(X)$ is a monic irreducible factor of $X^n -1$, we may take $f_1(X) = f(X)$ in the factorization (\ref{eq:factorization}).

Let $u(X) \in V_1 \leq \mathbb{F}_q[X]/\langle X^n -1 \rangle$ such that $u(X) \equiv 1 \text{ mod } f(X) $. We claim that the cyclic orbit of $u(X)$ (i.e., its orbit under repeated multiplication by $X$) is equal to $V_1 \setminus \{0\}$.

To prove this, we first observe that $X^m u(X) \neq u(X)$ for all $1 \leq m \leq n-1$. Indeed, suppose for a contradiction there exists $m \in [n-1]$ such that multiplication by $X^m$ fixes $u(X)$ in $\mathbb{F}_q [X]/\left \langle X^n -1 \right \rangle$. Then $X^m u(X) \equiv u(X) \text{ mod }(X^n -1)$, and therefore $X^n-1$ divides $(X^m -1)u(X)$. It follows that $\omega$ is a root of $(X^m -1)u(X)$. Since $\omega$ is a primitive $n$th root of unity, we have $\omega^m - 1 \neq 0$, and therefore $u(\omega)=0$. Hence, as $f(X)$ is the minimal polynomial of $\omega$, $f(X)$ divides $u(X)$. This contradicts our assumption that $u(X) \equiv 1 \text{ mod } f(X)$.

It follows that $u(X), Xu(X), X^2 u(X),...,X^{n-1} u(X)$ are $n$ distinct elements of $V_1 \setminus \{0\}$. Since $\textrm{dim}(V_1) = \textrm{deg}(f(X)) = d$, we have $|V_1 \setminus \{0\}| = q^{d} - 1 = n$. It follows that the cyclic orbit of $u(X)$ is precisely $V_1 \setminus \{0\}$, as claimed.

Let $U = \text{Span}\{u(X)\} \leq V_1$. Clearly, by the preceding claim, $U$ cyclically covers $V_1$. Note that the codimension of $U$ as a subspace of $V_1$ is $\text{dim}(V_1) - 1 = d-1$.

Finally, we set
$$U' = U \oplus \left(\bigoplus_{i =2}^{N} {V_i}\right) \leq \frac{\mathbb{F}_q [X]}{\left \langle X^n -1 \right \rangle},$$
and we claim that $U'$ cyclically covers $\mathbb{F}_q [X]/\left \langle X^n -1 \right \rangle$. Indeed, given $v(X) \in \mathbb{F}_q [X]/ \langle X^n -1 \rangle$, there exist unique $v_i(X) \in V_i$ (for $i=1,2,\ldots,N$) such that $v(X) = \sum_{i=1}^{N} v_i (X)$. Since $U$ cyclically covers $V_1$, there exists $m \in \{0,1,2,...,n-1\}$ such that $X^m v_1(X) \in U$. Since $X^m v_i (X) \in V_i$ for all $i \in [N]$, it follows that $X^m v(X) = \sum_{i=1}^{N} {X^m v_i (X)} \in U'$. Hence, $U'$ cyclically covers $\mathbb{F}_q [X]/ \langle X^n -1 \rangle$. Since the codimension of $U$ in $V_1$ is equal to $d-1$, the codimension of $U'$ in $\mathbb{F}_q [X]/ \langle X^n -1 \rangle$ is also equal to $d-1$.

It follows that $h_q(n) \geq d-1$. In combination with the upper bound in Lemma \ref{lemma:simple-upper}, this completes the proof of the theorem.

\end{proof}

With only a little extra work, the argument in the proof of Theorem \ref{thm:main} can be extended to obtain the following more general lower bound.

\begin{theorem}
\label{thm:general}
Let $q$ be a prime power and let $k,d \in \mathbb{N}$. Let $M = (q-1)(\sum_{r=0}^{d}{q^{kr}}) = (q-1)(q^{kd+k}-1)/(q^k-1)$, and suppose that $M$ has a divisor $c \in \mathbb{N}$ such that $c< (q-1) \frac{q^k - q^{-kd}}{q^k - 1}$.
Then
\begin{equation*}
h_q(M / c) \geq kd + k - c(q^k-1)/(q-1).
\end{equation*}
\end{theorem}

\begin{proof}
Let $q$ be a prime power, let $k,d \in \mathbb{N}$ and let $M = (q-1)(\sum_{r=0}^{d}{q^{kr}})$. Let $c \in \mathbb{N}$ be a divisor of $M$ satisfying $c< (q-1) \frac{q^k - q^{-kd}}{q^k - 1}$. Set $n := M/c$. As in the proof of Theorem \ref{thm:main}, we identify $\mathbb{F}_q^n$ with $\mathbb{F}_q[X]/\left \langle X^n - 1 \right \rangle$, and we decompose the latter into invariant subspaces,
$$\mathbb{F}_q[X]/\left \langle X^n - 1 \right \rangle = \bigoplus_{i = 1}^{N}{V_i},$$
by taking a factorisation
$$X^n - 1 = \prod_{i=1}^{N}{f_i(X)}$$
of $X^n-1$ into a product of irreducible monic factors, and taking $V_i$ to be the preimage of $\mathbb{F}_q[X]/\left \langle f_i(X) \right \rangle$ under the direct sum of the natural quotient maps, $p(X) \mapsto p(X) \text{ mod } f_i(X)$.

Note that the $V_i$ are irreducible subspaces. Indeed, suppose that $\{0\} \neq W \leq V_i$ and that $W$ is invariant under multiplication by $X$. Let $p(X) \in \mathbb{F}_q[X]$ such that the image of $p$ (under the natural quotient map $\mathbb{F}_q[X] \to \mathbb{F}_q[X]/\langle X^n -1 \rangle$) lies in $W \setminus \{0\}$. Then $f_i(X)$ does not divide $p(X)$, and $f_i(X)$ is irreducible, so $p(X)$ and $f_i(X)$ are coprime. Hence, by B\'ezout's lemma, there exist $s(X),t(X) \in \mathbb{F}_q[X]$ such that $s(X)p(X) + t(X)f_i(X) = 1$. Let $q(X) := s(X)p(X)$; we have $q(X) \equiv 1 \text{ mod } f_i(X)$. The invariance of $W$ under multiplication by $X$ implies that $q(X) \in W$, and moreover that

\begin{equation}
	\{q(X),Xq(X),X^2q(X),\ldots \} \subseteq W
\end{equation}
But the set on the left-hand side contains a basis for $V_i$, since for each $0 \leq r \leq n-1$ we have $X^r q(X) \equiv X^r \text{ mod } f_i(X)$. It follows that that $W = V_i$, proving the irreducibility of $V_i$.

We now continue to follow the proof of Theorem \ref{thm:main}. Let $\omega \in \overline{\mathbb{F}}_q$ be a primitive $n$th root of unity. We claim that the order of $q$ modulo $n$ is $k(d+1)$. Indeed, let $L$ be the order of $q$ modulo $n$. Since $nc(\sum_{t=0}^{k-1}{q^t}) = q^{k(d+1)} - 1$, we have $q^{k(d+1)} \equiv 1 \text{ mod } n$, and therefore $L$ divides $k(d+1)$. Since $q^{kd} < n \Leftrightarrow c < (q-1) \frac{q^k - q^{-kd}}{q^k - 1}$, we have $L>kd$. Since $kd \geq \tfrac{1}{2}k(d+1)$, no non-trivial factor of $k(d+1)$ is greater than $kd$, and therefore $L=k(d+1)$, as claimed. By Proposition \ref{prop:min-poly}, the minimal polynomial of $\omega$ over $\mathbb{F}_q$ is
\begin{equation*}
	f(X) = (X- \omega)(X- \omega^q)(X- \omega^{q^2})...(X-\omega^{q^{k(d+1) - 1}}) \in \mathbb{F}_q[X],
\end{equation*}
which has degree $k(d+1)$. We may assume without loss of generality that $f_1(X) = f(X)$, and consider $V_1$. As in the proof of Theorem \ref{thm:main}, let $u(X) \in V_1 \leq \mathbb{F}_q[X]/\langle X^n-1\rangle$ such that $u(X) \equiv 1 \text{ mod } f_1(X)$, and recall from the proof of Theorem \ref{thm:main} that $u(X)$ has orbit (under repeated multiplication by $X$) of size exactly $n$. More generally, let $0 \neq v(X) \in V_1$; we claim that the orbit 
\begin{equation*}
\{v(X),Xv(X),X^2 v(X),..., X^{n-1}v(X)\} \subseteq V_1
\end{equation*}
also has size exactly $n$. Indeed, since $V_1$ is irreducible, and 
\begin{equation*}
0 \neq \text{Span}(\{v(X),Xv(X),X^2 v(X),..., X^{n-1}v(X)\}) \leq V_1
\end{equation*}
is a subspace that is invariant under multiplication by $X$, we see that $\{v(X),Xv(X),X^2 v(X),..., X^{n-1}v(X)\}$ spans $V_1$. Suppose for a contradiction there exists $1 \leq a \leq n-1$ such that $X^a v(X) = v(X)$ (note that this is an equality in $\mathbb{F}_q[X]/\langle X^n -1 \rangle$). We may then express $u(X)$ as a linear combination $u(X) = \sum_{i=0}^{a-1}{\lambda_i X^i v(X)}$, for some $\lambda_i \in \mathbb{F}_q$. But then
\begin{equation*}
 X^a u(X) = \sum_{i=0}^{a-1}{\lambda_i X^i X^a v(X)} = \sum_{i=0}^{a-1}{\lambda_i X^i v(X)} = u(X)
\end{equation*}
contradicting the fact that the orbit of $u(X)$ has size exactly $n$. It follows that $v(X)$ has orbit of size exactly $n$, as claimed.

We may conclude all the orbits (under repeated multiplication by $X$) in $V_1 \setminus \{0\}$ have size $n$. There are $s := (|V_1| - 1) /n = (q^{k(d+1)} - 1)/n = c\sum_{t= 0}^{k-1}{q^t}$ such orbits; let $\{u_1,u_2,...,u_s\}$ be a set of representatives of these orbits. Then $U = \text{Span}(\{u_1,u_2,...,u_s\}) \leq V_1$ cyclically covers $V_1$, and has codimension (in $V_1$) at least $k(d+1) - s$.

Finally, we set
\begin{equation*}
U' = U \oplus \left( \bigoplus_{i \neq 1}{V_i} \right)  \leq \frac{\mathbb{F}_q [X]}{\left \langle X^n - 1 \right \rangle};
\end{equation*}
note that $U'$ cyclically covers $\mathbb{F}_q [X]/\langle X^n - 1 \rangle$ and has codimension (in $\mathbb{F}_q [X]/\langle X^n - 1 \rangle$) at least $k(d+1) - s = kd +k - c\sum_{t=0}^{k-1}{q^t} = kd+k - c(q^k-1)/(q-1)$. It follows that $h_q(n) \geq kd+k - c(q^k-1)/(q-1)$, as required.

\end{proof}

Applying the above theorem with $c = 1$, fixed $q,k$ and $d \to \infty$, and appealing to Lemma \ref{lemma:simple-upper}, we see that
\begin{equation*}
h_q\left((q-1)\sum_{r=0}^{d}{q^{kr}}\right) = (1+o(1))kd
\end{equation*}
where the $o(1)$ term tends to zero as $d$ tends to infinity. Theorem \ref{thm:main} is recovered from Theorem \ref{thm:general} by setting $k = 1$ and $c=1$.

We now demonstrate how a slight variation on the ideas in the proofs of Theorem \ref{thm:main} and Theorem \ref{thm:general} can determine $h_q(n)$ for other infinite sequences of integers $n$ (for each fixed prime power $q$).

\begin{theorem}
\label{thm:pjc}
Let $q$ be a prime power, and let $k,d \in \mathbb{N}$ such that $\text{gcd}(d+1,q^k - 1) = 1$. Set $n = \sum_{r=0}^{d}{q^{kr}} = \frac{q^{k(d+1)}-1}{q^k-1}$. Then
\[h_q(n)=kd.\]
\end{theorem}

\begin{proof}
The upper bound $h_q(n)\leq kd$ follows immediately from Lemma \ref{lemma:simple-upper}, so we need
only prove the lower bound. We first note that 
$$n =\sum_{r=0}^{d}{q^{kr}} \equiv d+1 \ \text{mod } q^k-1,$$
since $q^{kr} \equiv 1$ mod $q^k-1$ for each $r \in \mathbb{N} \cup \{0\}$, so
\begin{equation}\label{eq:cong} \text{gcd}(n,q^k-1) = \text{gcd}(d+1,q^k-1)=1.\end{equation}

As in the proofs of Theorems \ref{thm:main} and \ref{thm:general}, we identify $\mathbb{F}_q^n$ with $\mathbb{F}_q[X]/\left \langle X^n - 1 \right \rangle$, and we decompose the latter into invariant subspaces,
$$\mathbb{F}_q[X]/\left \langle X^n - 1 \right \rangle = \bigoplus_{i = 1}^{N}{V_i},$$
by taking a factorisation
$$X^n - 1 = \prod_{i=1}^{N}{f_i(X)}$$
of $X^n-1$ into a product of irreducible monic factors, and taking $V_i$ to be the preimage of $\mathbb{F}_q[X]/\left \langle f_i(X) \right \rangle$ under the direct sum of the natural quotient maps.

Let $\omega \in \overline{\mathbb{F}}_q$ be a primitive $n$th root of unity. As in the proof of Theorem \ref{thm:general}, we claim that $q$ has multiplicative order $k(d+1)$ modulo $n$. Indeed, let $L$ be the order of $q$ modulo $n$. Since $q^{k(d+1)} - 1 = n (q^k - 1)$, we have $q^{k(d+1)} \equiv 1 \text{ mod }n$, and therefore $L$ divides $k(d+1)$. Since $q^{kd} < n$, we must have $L>kd$. Since $kd \geq \tfrac{1}{2}k(d+1)$, no non-trivial factor of $k(d+1)$ is greater than $kd$, and therefore $L=k(d+1)$, as claimed. By Proposition \ref{prop:min-poly},
the minimal polynomial of $\omega$ over $\mathbb{F}_q$ is
\begin{equation*}
	f(X) = (X- \omega)(X- \omega^q)(X- \omega^{q^2})...(X-\omega^{q^{k(d+1) - 1}}) \in \mathbb{F}_q[X],
\end{equation*}
which has degree $k(d+1)$. We may assume without loss of generality that $f_1(X) = f(X)$, and consider $V_1$. Since $V_1 \cong \mathbb{F}_q[X]/\langle f(X) \rangle$ and $f(X)$ is an irreducible polynomial of degree $k(d+1)$, $V_1$ is in fact a field extension of $\mathbb{F}_q$ (of degree $k(d+1)$), and as such can be identified with the finite field $\mathbb{F}_{q^{k(d+1)}}$. Hence, $V_1$ can also be viewed as a $(d+1)$-dimensional vector space over (a field isomorphic to) $\mathbb{F}_{q^k}$.

Let $U$ be a 1-dimensional $\mathbb{F}_{q^k}$-subspace of $V_1$. We now make two claims regarding $U$. Firstly, we claim that no power of the shift map can map $U$ to itself. Indeed, suppose for a contradiction that there exists $a \in [n-1]$ such that $X^a U=U$. Then for any $u \in U \setminus \{0\}$, we have
\begin{equation*}
	\{u,X^a u, X^{2a}u,\ldots\} \subseteq U.
\end{equation*}

We note, as in the proof of Theorem \ref{thm:general}, that for every $v \in V_1 \setminus \{0\}$, the orbit of $v$ under repeated multiplication by $X$ has size $n$. Hence, for $j \in \mathbb{N}$, $X^j v = v$ if and only if $n \mid j$. It follows that for any $u \in U \setminus \{0\}$, we have $X^{aj} u = u$ if and only if $n \mid aj$, i.e.\ if and only if $n/\text{gcd}(a,n) \mid j$. Therefore, the above orbit of $u$ under repeated multiplication by $X^a$ has size exactly $n/\text{gcd}(a,n) :=M$. The family of all such orbits (of non-zero elements of $U$, under repeated multiplication by $X^a$) partitions $U \setminus \{0\}$ into sets of equal size $M$, and therefore $M$ is a proper divisor of $n$ that also divides $|U|-1 = q^k-1$. But this contradicts (\ref{eq:cong}).

Secondly, we claim that $X^b U \cap X^c U = \{0\}$ for any $0 \leq b < c \leq n-1$. Indeed, suppose for a contradiction that there exist $0 \leq b < c \leq n-1$ such that $X^b U \cap X^c U \neq \{0\}$. Then, multiplying by $X^{n-b}$, we have $X^{n}U \cap X^{n+c-b}U \neq \{0\}$ and therefore $U \cap X^a U \neq \emptyset$, where $a: = c-b \in [n-1]$. However, $U$ and $X^a U$ are distinct 1-dimensional $\mathbb{F}_{q^k}$-subspaces of $V_1$, so have intersection $\{0\}$, a contradiction.

It follows that $U,XU,\ldots,X^{n-1}U$ are $q^k$-element subsets of $V_1$ whose pairwise intersections are all equal to $\{0\}$; since $q^{k(d+1)}-1 = n (q^k-1)$, we must have $V_1 = \cup_{a=0}^{n-1}X^a U$, so $U$ (as an $\mathbb{F}_q$-subspace of $V_1$) is a cyclic cover of $V_1$ with codimension $kd$ in $V_1$.

Set $U' = U \oplus (\oplus_{i \neq 1}{V_i})$; then $U'$ is a cyclic cover of $\mathbb{F}_q [X]/\langle X^n - 1 \rangle$ with codimension $kd$, so $h_q(n) \geq kd$, as required.
\end{proof}

We note that for a fixed prime power $q$ and a fixed integer $k$, a positive fraction of positive integers $d$ have the property that $\text{gcd}(d+1,q^{k}-1)=1$ (and so satisfy the hypothesis of Theorem \ref{thm:pjc}).

\subsubsection*{Implications for the lower bound in Isbell's conjecture and the generalised Isbell conjecture}

The $q=2$ cases of Theorems \ref{thm:main} and \ref{thm:pjc}, together with the $p=2$ case of Proposition \ref{prop:isbell-connection}, imply the following.
\begin{corollary}
\label{corr:isbell}
For any $d \in \mathbb{N}$, we have $h_2(2^d-1)=d-1$, and therefore $m(2^d-1) \geq d$. Moreover, for any $d,k \in \mathbb{N}$ with $\text{gcd}(d+1,2^{k}-1)=1$, we have $h_2((2^{k(d+1)}-1)/(2^k-1)) = kd$, and therefore $m((2^{k(d+1)}-1)/(2^k-1)) \geq kd+1$. 
\end{corollary}

We remark that a construction due to Spiga \cite{spiga-thesis} (building on the work of Suzuki \cite{suzuki} in which he introduced and analysed the Suzuki groups), gives
$$m((2^{kr}-1)/(2^k-1)) \geq k(r-1)^2+1$$
for all primes $r > 2$ and integers $k \in \mathbb{N}$ coprime to $2^r-1$. In particular, $m(2^r-1) \geq (r-1)^2+1$ for all primes $r >2$, giving a lower bound on $m(b)$ that is quadratic in $\log b$ for infinitely many $b$, whereas our lower bound in Corollary \ref{corr:isbell} is only linear in $\log b$. Spiga's construction involves replacing the Abelian group $V$ (in the penultimate paragraph of the Introduction) with a non-Abelian group $N$ (for example, a certain Sylow 2-subgroup of a Suzuki group), and finding a subgroup of $N$ of large index, whose images under an appropriate cyclic automorphism cover $N$. However, our construction is simpler and provides good lower bounds on the function $m$ for other natural infinite sequences of odd integers, where Spiga's method does not apply.

Similarly, the $q=p$ cases of Theorems \ref{thm:main} and \ref{thm:pjc}, together with the general case of Proposition \ref{prop:isbell-connection}, imply the following.

\begin{corollary}
\label{corr:isbell-gen}
For any $d \in \mathbb{N}$, we have $h_p(p^d-1)=d-1$, and therefore $m_p(p^d-1) \geq d$. Moreover, for any $d,k \in \mathbb{N}$ with $\text{gcd}(d+1,p^{k}-1)=1$, we have $h_p((p^{k(d+1)}-1)/(p^k-1)) = kd$, and therefore $m_p((p^{k(d+1)}-1)/(p^k-1)) \geq kd+1$. 
\end{corollary}

In this case, Spiga's construction in \cite{spiga-thesis} yields 
$$m_p((p^{kr}-1)/(p^k-1)) \geq k(r-1)^2 +1$$
for all primes $r \neq p$ and $k \in \mathbb{N}$ such that $r$ and $p^k-1$ are coprime. In particular, $m_p((p^r-1)/(p-1)) \geq (r-1)^2+1$ for all primes $r \neq p$ such that $r$ does not divide $p-1$. Again, for a fixed prime $p$, this gives a lower bound on $m_p(b)$ that is quadratic in $\log b$ for infinitely many $b$, whereas our lower bound in Corollary \ref{corr:isbell-gen} is only linear in $\log b$; but again, our construction is simpler and works in cases where Spiga's method does not apply.

\section{General representations of groups}

In this section, we generalise our discussion to arbitrary group representations. The proofs of the bounds in this section are straightforward, given some basic facts from the representation theory of finite groups; nevertheless, each bound is sharp in some non-trivial cases (in fact, infinitely many).

Let $G$ be a group, let $\mathbb{F}$ be a field and let $V$ be a vector space over $\mathbb{F}$. We write $\text{GL}(V)$ for the general linear group of $V$. Let $\rho: G \rightarrow \text{GL}(V)$ be a group homomorphism, i.e.\ $(\rho,V)$ is a representation of $G$. Let us say that a subspace $U \leq V$ is {\em $(G,\rho)$-covering} if
\begin{equation*}
\bigcup_{g \in G}{\rho(g)(U)} = V
\end{equation*}
where $\rho(g)(U) := \{\rho(g)(u) :\ u \in U\}$. Let us define $h_{G,\rho}(V)$ to be the maximum possible codimension of a $(G,\rho)$-covering subspace of $V$. Note that $h_q(n) = h_{C_n,\rho_{\sigma}}(\mathbb{F}_q^n)$ for any prime power $q$ and any $n \in \mathbb{N}$, where $(\rho_{\sigma},\mathbb{F}_q^n)$ is the representation of $C_n$ that maps the generator of $C_n$ to $\sigma$.

Let us briefly outline the representation-theoretic terminology and notation we will use. As usual, from now on we will sometimes write $g(u)$ in place of $\rho(g)(u)$, when the representation $\rho$ is understood. Recall that if $(\rho,V)$ is a fixed representation of $G$, a subspace $W \leq V$ is said to be {\em $G$-invariant} if $\rho(g)(w) \in W$ for any $w \in W$ and any $g \in G$; in this case, $(\rho,W)$ is said to be a {\em subrepresentation} of $(\rho,V)$. Abusing terminology slightly, when $\rho$ is understood, we will sometimes omit it from our notation, and describe $W$ as a subrepresentation of $V$.

If $G$ is a finite group, $q$ is a prime power, $V$ is a finite dimensional vector space over $\mathbb{F}_q$ and $(\rho,V)$ is a representation of $G$, it is easy to obtain the upper bound
\begin{equation}\label{eq:gen-orbit-counting} h_{G,\rho}(V) \leq \left \lfloor\log_q {|G|} \right \rfloor,\end{equation}
just as in the proof of Lemma \ref{lemma:simple-upper}, since the group $G$ acts on $V$, partitioning $V$ into orbits, each of size at most $|G|$.

Turning to general lower bounds, the following is easy to obtain.

\begin{lemma}
\label{lemma:lower-gen}
Let $\mathbb{F}$ be a field, let $G$ be a finite group such that $\textrm{char}(\mathbb{F}) \nmid |G|$, let $V$ be a vector space over $\mathbb{F}$, and let $(\rho,V)$ be a representation of $G$. If $W$ is a subrepresentation of $V$, then
\begin{equation*}
h_{G,\rho}(W) \leq h_{G,\rho}(V).
\end{equation*}
\end{lemma}

\begin{proof}
As in the standard proof of Maschke's theorem, we can find a $G$-invariant subspace $W'$ of $V$ such that $V= W \oplus W'$. (Let $\{w_1,...,w_k\}$ be a basis for $W$, and extend it to a basis $\{w_1,...,w_m\}$ for $V$. Define $\pi : V \rightarrow W$ by $\pi (\sum_{i=1}^{m} {\lambda_i w_i}) = \sum_{i=1}^{k} {\lambda_i w_i}$, and define
\begin{equation*}
\bar{\pi} (w) = \frac{1}{|G|} \sum_{g \in G}{g^{-1} (\pi (g (w)))}.
\end{equation*}
Let $W' = \text{ker}(\bar{\pi})$; then $V = W \oplus W'$, and $W'$ is $G$-invariant, since if $v \in W'$ and $g \in G$ then $\bar{\pi} (g (v) ) = g(\bar{\pi}(v)) = 0$, so $g (v) \in W'$.)

Now let $Z \leq W$ be a $(G,\rho)$-covering subspace of $W$ with codimension $h_{G,\rho}(W)$, and let $U = Z \oplus W' \leq V$. It is easy to see that $U$ is a $(G,\rho)$-covering subspace of $V$; clearly, its codimension in $V$ is the same as that of $Z$ in $W$. Hence, $h_{G,\rho}(V) \geq \text{codim} (U) = \text{codim} (Z) = h_{G,\rho}(W)$, proving the lemma.
\end{proof}

We also have the following easy general upper bound.

\begin{lemma}
\label{lemma:upper-gen}
Let $G$ be a group, let $(\rho,V)$ be a representation of $G$, and suppose that $V = \bigoplus_{i} W_i$, where each $W_i$ is a $G$-invariant subspace of $V$. Then
$$h_{G,\rho}(V) \leq \sum_{i} h_{G,\rho}(W_i).$$
\end{lemma}

\begin{proof}
Let $U$ be a $(G,\rho)$-covering subspace of $V$ with codimension $h_{G,\rho}(V)$, and let $W_{i}^{*} = W_i \cap U$ for each $i$. Then $W_{i}^{*}$ is clearly a $(G,\rho)$-covering subspace of $W_i$, and therefore $\text{codim}(W_{i}^{*}) \leq h_{G,\rho}(W_i)$. 

We have $\bigoplus_i {W_{i}^{*} \leq U \leq V}$, and therefore
$$h_{G,\rho}(V) = \text{codim}(U) \leq \text{codim}\left(\bigoplus_i W_{i}^{*}\right) = \sum_i {\text{codim}(W_{i}^{*})} \leq \sum_i {h_{G,\rho}(W_i)},$$
proving the lemma.
\end{proof}

The following is an immediate corollary of (\ref{eq:gen-orbit-counting}) and Lemmas \ref{lemma:lower-gen} and \ref{lemma:upper-gen}.
\begin{corollary}
\label{cor:combined-bounds}
Let $q$ be a prime power, let $G$ be a finite group with order coprime to $q$, let $V$ be a finite-dimensional vector space over $\mathbb{F}_q$, and let $(\rho,V)$ be a representation of $G$. Let $V = \bigoplus_{i}{W_i}$ be a decomposition of $V$ into subrepresentations. Then
\begin{equation}
\label{eq:combined}
\max_{i}\{h_{G,\rho}(W_i)\} \leq h_{G,\rho}(V) \leq \min{\{\sum_{i} {h_{G,\rho}(W_i)},\left \lfloor\log_q(|G|) \right \rfloor \}}
\end{equation}
\end{corollary}
We remark that, under the hypotheses of Corollary \ref{cor:combined-bounds}, Maschke's theorem guarantees the existence of a decomposition of $V = \bigoplus_{i}{W_i}$ where each $W_i$ is an irreducible subrepresentation of $V$.

Theorem \ref{thm:main} implies that for the rotation action $\rho_{\sigma}$ of $C_{q^d-1}$ on $\mathbb{F}_q^{q^d-1}$, the lower bound in (\ref{eq:combined}) is tight, as is the $\left \lfloor \log_q (|G|) \right \rfloor$ upper bound, for all $d \in \mathbb{N}$. We remark that there are infinitely many (nontrivial) cases where the sum bound is sharp, and distinct from both the lower bound and the $\log_q (|G|)$ bound. Indeed, let $m_1,...,m_k \in \mathbb{N}$ be chosen such that $\text{gcd}(2^{m_i} - 1,2^{m_j} - 1) = 1$ for all $i \neq j$. (For infinitely many examples of such choices, one may take the $m_i$'s to be distinct primes, since if $d \mid 2^a -1$ and $d \mid 2^b - 1$ then $d \mid 2^{\text{gcd}(a,b)} - 1$.) For each $i\in[k]$, let $n_i = 2^{m_i}-1$, let $n = \prod_{i=1}^k n_i$, let $\omega_i \in \overline{\mathbb{F}}_2$ be a primitive $n_i$th root of unity and let $f_i(X) \in \mathbb{F}_2[X]$ be the minimal polynomial of $\omega_i$. As in the proof of Theorem \ref{thm:main}, $f_i(X)$ divides $X^{n_i}-1$ and has degree $m_i$. For each $i \in [k]$, let $W_i = \mathbb{F}_2[X]/\langle f_i(X) \rangle$ be the representation of $C_n$ where the generator of $C_n$ acts by multiplication by $X$; note that $\dim(W_i) = m_i$ for all $i$. Let
\begin{equation}
\label{eq:direct-sum}
V=\bigoplus_{i=1}^k W_i
\end{equation}
be the direct sum of these representations. By the Chinese Remainder Theorem for rings, as in the proof of Theorem \ref{thm:main}, $V$ can be identified with a subspace $\theta^{-1}(V)$ of $\mathbb{F}_2[X]/ \langle X^n - 1\rangle$, where $\theta$ is the linear isomorphism defined in (\ref{eq:theta}); under this identification, the generator of $C_n$ again acts by multiplication by $X$. It follows from the proof of Theorem \ref{thm:main} that $h_{C_n}(W_i) = m_i-1$ for each $i \in [k]$; indeed, the action of $C_n$ is transitive on $W_i\setminus \{0\}$. For each $i \in [k]$, let $u_i(X) \in \theta^{-1}(V)$ such that $u_i(X) \equiv 1$ modulo $f_i(X)$ and $u_i(X) \equiv 0$ modulo $f_j(X)$ for all $j \neq i$, and let $U = \text{span}\{u_1(X),\ldots,u_k(X)\} \leq V$. It is easy to see, using the Chinese Remainder Theorem (and the fact that $|W_i \setminus \{0\}| = n_i$ is coprime to $|W_j \setminus \{0\}| = n_j$ for $i \neq j$), that $U$ is a cyclically covering subspace of $\theta^{-1}(V)$; it has codimension
$$\dim(V)-k = \sum_{i=1}^{k}m_i-k = \sum_{i=1}^{k}(m_i-1) = \sum_{i = 1}^k h_{C_n}(W_i),$$
showing that the sum bound is sharp, for the direct sum in (\ref{eq:direct-sum}). It is easy to check that the sum bound here is also distinct from the $\log_q(|G|)$ bound, whenever $m_i \geq 2$ for all $i$.

\vspace{0.4cm}

In a forthcoming paper, the last two authors investigate the behaviour of $h_{S_n,\rho}(V)$, for various representations $(\rho,V)$ of the symmetric group $S_n$.

\section{Cases in which the covering subspaces are trivial}

In this section, we demonstrate the opposite behaviour to that seen in Theorems \ref{thm:main} and \ref{thm:pjc} for other sequences of integers.
\begin{theorem}
\label{thm:prime-power}
Let $p$ be a prime, let $q$ be a power of $p$, let $k \in \mathbb{N}$ with $k \mid q-1$, and let $d \in \mathbb{N}$. Then
\begin{equation*}
h_q(k p^d) = 0.
\end{equation*}
Equivalently, if $U \leq \mathbb{F}_q^{k p^d}$ is cyclically covering, then $U = \mathbb{F}_q^{k p^d}$.
\end{theorem}

In fact, Theorem \ref{thm:prime-power} is a special case of the following result for more general representations.
\begin{theorem}
\label{thm:prime-power-general}
Let $p$ be a prime. Let $G = A \times B$, where $A$ is an Abelian group of exponent $k$ dividing $q-1$, and $B$ is a finite $p$-group. Let $q$ be a power of $p$, let $V$ be a finite-dimensional vector space over $\mathbb{F}_q$ and let $(\rho,V)$ be a representation of $G$. Then $h_{G,\rho}(V)=0$.
\end{theorem}

Theorem \ref{thm:prime-power-general}, in turn, is a consequence of the following two lemmas, which may be of independent interest.
\begin{lemma}
\label{lemma:peter}
Suppose that a finite $p$-group $Q$ acts on a finite $p$-group $P$ by automorphisms. If
$H$ is a subgroup of $P$ such that $\displaystyle{\bigcup_{g\in Q}H^g=P}$,
then $H=P$. (Here, as usual, $H^g$ denotes the image of $H$ under the automorphism defined by $g$.)
\end{lemma}

\begin{proof} The proof is by induction on $|P|$. The result is clear
if $|P|=1$, so suppose that $|P|>1$ and that the result holds for all smaller $p$-groups.

Write $\Phi(P)$ for the Frattini subgroup of $P$, i.e.\ the intersection of all maximal subgroups of $P$. The number of subgroups of index $p$ in $P$ is equal to $(p^d-1)/(p-1)$, where $d \in \mathbb{N}$ is such that $|P/\Phi(P)|=p^d$. (This is well-known, and follows from the facts that a subgroup of
index $p$ in $P$ is normal, and $\Phi(P)$ is the minimal normal subgroup of
$P$ with elementary abelian quotient; see e.g.\ \cite[1.D.8]{isaacs}). Observe that $Q$ acts by automorphisms on the set of all index-$p$ subgroups of $P$. Since $Q$ is a $p$-group, the orbit-stabilizer theorem implies that every orbit of this action has size a power of $p$. Since the number of index-$p$ subgroups, $(p^d-1)/(p-1)$, is coprime to $p$, one of these orbits has size one. In other words, some index-$p$ subgroup, $P_1$ say, is fixed by $Q$. Therefore,
\[P_1=\bigcup_{g\in Q}(H\cap P_1)^g.\]
By the induction hypothesis, $P_1\cap H=P_1$, so $P_1\le H$. Since $P_1$ is fixed by $Q$, we cannot have $H=P_1$. It follows that $H=P$, completing the induction step, proving the lemma.
\end{proof}

\begin{lemma}
\label{lemma:peterII}
Let $p$ be prime, and let $q$ be a power of $p$. Let $G=A\times B$, where $A$ is an Abelian group of 
exponent $k$ dividing $q-1$, and $B$ is a finite $p$-group. Let $G$ act linearly on a vector
space $V$ over $\mathbb{F}_q$, where the action by $B$ is by automorphisms of $V$, and let $U$ be a subspace of $V$ such that the
union of the images of $U$ under $G$ cover $V$. Then $U=V$.
\end{lemma}
\begin{proof}
In the case $k=1$, this follows from Lemma \ref{lemma:peter}, applied with $Q=B$ and $P$ the additive group of $\mathbb{F}_q$ (noting that the representation action of $B$ on $V$ corresponds to an action on $P$ by automorphisms). Suppose then that $k > 1$.

Every element of $A$ (viewed as a linear endomorphism of $V$) has minimum polynomial dividing $X^k-1$, which has
$k$ distinct roots in $\mathbb{F}_q$ (since $X^{k}-1$ divides $X^{q-1}-1$, which has $q-1$ distinct roots in $\mathbb{F}_q$), so the elements of $A$ are all
diagonalisable. Since $A$ is Abelian, its elements can be simultaneously
diagonalised, so $V$ is the direct sum of the common eigenspaces. Since $B$
commutes with $A$, it fixes each of these eigenspaces, and therefore so does $G$. If $U$ contains all
of the eigenspaces, then we have $U=V$, as required. Hence, we may assume that 
$U\cap W\subset W$ for some eigenspace $W$. Then the images under $G$ of
$U\cap W$ cover $W$. However, every element of $A$ acts as a scalar on $W$,
and so fixes every subspace of $W$. So the images of $U\cap W$ under $B$ cover $W$.
The result for $k=1$ now implies that $U\cap W=W$, contrary to our assumption.
\end{proof} 

\begin{proof}[Proof of Theorem \ref{thm:prime-power-general}.]
Let $p$ be a prime. Let $G = A \times B$, where $A$ is an Abelian group of exponent $k$ dividing $q-1$, and $B$ is a finite $p$-group. Let $q$ be a power of $p$, let $V$ be a finite-dimensional vector space over $\mathbb{F}_q$ and let $(\rho,V)$ be a representation of $G$.  Let $U \leq V$ such that $\bigcup_{g \in G} \rho(g)(U) = V$. By Lemma \ref{lemma:peterII}, we must have $U=V$, proving the theorem.
\end{proof}

Theorem \ref{thm:prime-power} follows quickly from Theorem \ref{thm:prime-power-general}.
\begin{proof}[Proof of Theorem \ref{thm:prime-power}.]
If $k \mid q-1$, then $(k,p)=1$ so $C_{kp^d} \cong C_k \times C_{p^d}$. The group $C_k$ has exponent $k$, and the group $C_{p^d}$ is a $p$-group, so we can apply Theorem \ref{thm:prime-power-general} with $V=\mathbb{F}_q^n$, yielding Theorem \ref{thm:prime-power}.
\end{proof}

We remark that Theorem \ref{thm:prime-power}, combined with some of our previous lemmas, determines completely the zeros of $h_2$.

\begin{corollary}
We have $h_2 (n) = 0$ if and only if $n = 2^d$ for some $d \in \mathbb{N} \cup \{0\}$, and $h_2(n)=1$ if and only if $n=3$.
\end{corollary}
\begin{proof}
Applying Theorem \ref{thm:prime-power} with $k=1$ yields $h_2(2^d) = 0$ for all $d \in \mathbb{N}$. Trivially, $h_2(1)=0$, and it is easy to see that $h_2(3)=1$. If $n>3$ and $n$ is not a power of 2, then $n$ is either divisible by 6 or by some odd number greater than $3$. Let $m$ be such a divisor. Lemma \ref{lemma:simple} implies that $h_2(m) \geq 2$ for all odd $m > 3$, and it can be checked that $h_2(6)=2$. Hence, by Lemma \ref{lma:products}, we have $h_2(n) \geq h_2(m) \geq 2$, proving the corollary.
\end{proof}

It would be interesting to determine completely, for each prime power $q>2$, the set $\{n \in \mathbb{N}:\ h_q(n)=0\}$. We remark that there are other zeros of $h_3$ besides $\{k3^d:\ k \in \{1,2\},\ d \in \mathbb{N}\}$ (those given by Theorem \ref{thm:prime-power}); for example, $h_3(4)=0$.

\section{Conclusion}
For each prime power $q$, we have found infinitely many values of $n$ such that $h_q(n) = \lfloor \log_q(n) \rfloor$ (these values of $n$ forming certain geometric series with common ratio $q$ or a power of $q$), and also infinitely many values of $n$ such that $h_q(n) =0$ (these values of $n$ forming geometric progressions $(k p^d)_{d \in \mathbb{N}}$, where $q$ is a power of the prime $p$ and $k \mid q-1$). This demonstrates that the behaviour of $h_q(n)$ as a function of $n$ is very irregular, depending heavily upon the prime factorization of $n$. It would be interesting to determine more precisely the behaviour of $h_q(n)$ for $n$ not of these forms. We remark again that the original question of the first author, as to whether $h_2(n)$ tends to infinity as $n$ tends to infinity over odd integers $n$, remains open, as does Isbell's conjecture.
\vspace{0.4cm}

\noindent {\em Note added after peer review:} Some time after a preprint of this paper first appeared on arXiv, Aaronson, Groenland and Johnston \cite{agj} proved that $h_2(p)=2$ for any prime $p$ such that $2$ is a primitive root modulo $p$; conditional on Artin's conjecture (that there are infinitely many such primes $p$), this answers the original question of the first author in the negative. We remark that Artin's conjecture is widely believed; indeed, as shown by Hooley \cite{hooley}, it is implied by the Generalized Riemann Hypothesis. However, the first author's question remains open in the strict (i.e., unconditional) sense. In \cite{agj}, Aaronson, Groenland and Johnston prove several other elegant results on the problems we consider above; for example, they prove that if $q$ is an odd prime, and $p > q$ is a prime such that $q$ is a primitive root modulo $p$, then $h_q(p)=0$. It is known, by a result of Heath-Brown \cite{hb}, that for all but at most two primes $q$, there are infinitely many primes $p$ for which $q$ is a primitive root modulo $p$, so the aforementioned result of Aaronson, Groenland and Johnson implies (unconditionally) that for all but at most two odd primes $q$, $h_q(n)$ does not tend to infinity as $n$ tends to infinity over integers coprime to $q$.

\subsubsection*{Acknowledgements}
We would like to thank Pablo Spiga for pointing out the references \cite{spiga-thesis,suzuki}, and Alex Fink for useful comments after a seminar on an early version of this paper. We would also like to thank two anonymous referees for their careful reading of the paper, and for their helpful suggestions.

\end{document}